\documentclass[reqno,12pt]{amsart}
\usepackage{color}

\theoremstyle{plain}
\newtheorem{thm}{Theorem}[section]
\newtheorem{cor}[thm]{Corollary} %%Delete [thm] to re-start numbering
\newtheorem{lemma}[thm]{Lemma} %%Delete [thm] to re-start numbering
\newtheorem{prop}[thm]{Proposition}

\theoremstyle{remark}

\theoremstyle{definition}

\newtheorem{example}[thm]{Example}

\newtheorem{ques}[thm]{Question}

\usepackage{amscd,amssymb,comment,epic,eepic,euscript,graphics}
\usepackage{enumerate}
\usepackage[initials]{amsrefs}
\newcount\theTime
\newcount\theHour
\newcount\theMinute
\newcount\theMinuteTens
\newcount\theScratch
\theTime=\number\time
\theHour=\theTime
\divide\theHour by 60
\theScratch=\theHour
\multiply\theScratch by 60
\theMinute=\theTime
\advance\theMinute by -\theScratch
\theMinuteTens=\theMinute
\divide\theMinuteTens by 10
\theScratch=\theMinuteTens
\multiply\theScratch by 10
\advance\theMinute by -\theScratch
\def\timeHHMM{{\number\theHour:\number\theMinuteTens\number\theMinute}}

\def\today{{\number\day\space
 \ifcase\month\or
  January\or February\or March\or April\or May\or June\or
  July\or August\or September\or October\or November\or December\fi
 \space\number\year}}

\def\timeanddate{{\timeHHMM\space o'clock, \today}}
\newcommand\Ac{{\mathcal{A}}}
\newcommand\Afr{{\mathfrak A}}

\newcommand\alphat{{\tilde\alpha}}

\newcommand\Bh{{\widehat B}}
\newcommand\Ch{{\widehat C}}
\newcommand\Comm{{\operatorname{Comm}}}
\newcommand\Cpx{{\mathbf C}}
\newcommand\diag{\text{\rm diag}}

\newcommand\Fb{{\mathbf F}}
\newcommand\HEu{{\EuScript H}}                   % requires package euscript
\newcommand\Ints{{\mathbf Z}}
\newcommand\kerproj{\operatorname{kerproj}}
\newcommand\Mcal{{\mathcal{M}}} %  needed to be renamed, because was ``\Mc already defined''
\newcommand\Nats{{\mathbf N}}
\newcommand\oneh{{\hat 1}}
\newcommand\ranproj{\operatorname{ranproj}}

\newcommand\tr{{\mathrm{tr}}}
\newcommand\Tt{{\widetilde T}}
\newcommand\Vh{{\widehat V}}
\newcommand\Xh{{\widehat X}}

\begin{document}

\title[single commutators]{On single commutators in II$_1$--factors}

\author[Dykema]{Ken Dykema$^{*}$}
\author[Skripka, \timeanddate]{Anna Skripka$^\dag$}
\address{K.D., Department of Mathematics, Texas A\&M University,
College Station, TX 77843-3368, USA}
\email{kdykema@math.tamu.edu}
\address{A.S., Department of Mathematics, University of Central Florida,
4000 Central Florida Blvd., P.O.\ Box 161364, Orlando, FL 32816-1364, USA}
\email{skripka@math.ucf.edu}
\thanks{\footnotesize ${}^{*}$Research supported in part by NSF grant DMS--0901220.
${}^{\dag}$Research supported in part by NSF grant DMS-0900870}



\date{February 1, 2011}

\begin{abstract}
We investigate the question of whether all elements of trace zero in a II$_1$--factor are
single commutators.
We show that all nilpotent elements are single commutators,
as are all normal elements of trace zero whose spectral distributions are discrete measures.
Some other classes of examples are considered.
\end{abstract}

\maketitle

\section{Introduction}

In an algebra $\Afr$, the {\em commutator} of $B,C\in\Afr$ is $[B,C]=BC-CB$, and we denote by $\Comm(\Afr)\subseteq\Afr$
the set of all commutators.
A {\em trace} on $\Afr$ is by definition a linear functional that vanishes on $\Comm(\Afr)$.
The algebra $M_n(k)$ of $n\times n$ matrices over a field $k$ has a unique trace, up to scalar multiplication;
(we denote the trace sending the identity element to $1$ by $\tr_n$).
It is known that every element of $M_n(k)$
that has null trace is necessarily a commutator (see~\cite{S36} for the case of characteristic zero and \cite{AM57} for the case of an arbitrary characteristic).
For the complex field, $k=\Cpx$, a natural generalization of the algebra $M_n(\Cpx)$
is the algebra $B(\HEu)$ of all bounded operators on a separable, possibly infinite dimensional Hilbert space $\HEu$.
Thanks to the ground breaking paper~\cite{BP65} of Brown and Pearcy,
$\Comm(B(\HEu))$ is known:
the commutators in $B(\HEu)$ are precisely the operators that are not of the form $\lambda I+K$ for $\lambda$
a nonzero complex number, $I$ the identity operator and $K$ a compact operator (and an
analogous result holds when $\HEu$ is nonseparable).

Characterizations of $\Comm(B(X))$ for some Banach spaces $X$ are found
in~\cite{A72}, \cite{A73}, \cite{D09} and~\cite{DJ10}.

The von Neumann algebra factors form a natural family of algebras including the matrix algebras $M_n(\Cpx)$
and $B(\HEu)$ for infinite dimensional Hilbert spaces $\HEu$; (these together are the type~I factors).
The set $\Comm(\Mcal)$ was determined by Brown and Pearcy~\cite{BP66} for $\Mcal$ a factor of type III
and by Halpern~\cite{H69} for $\Mcal$ a factor of type II$_\infty$.

The case of type II$_1$ factors remains open.
A type II$_1$ factor is a von Neumann algebra $\Mcal$ whose center is trivial and that has a trace $\tau:\Mcal\to\Cpx$,
which is then unique up to scalar multiplication;
by convention, we always take $\tau(1)=1$.
The following question seems natural, in light of what is known for matrices:
\begin{ques}\label{qn:comm}
Do we have
\[
\Comm(\Mcal)=\ker\tau
\]
for any one particular II$_1$--factor $\Mcal$, or even for all II$_1$--factors?
\end{ques}

Some partial results are known.
Fack and de la Harpe~\cite{FH80} showed that every element of $\ker\tau$ is a sum of ten commutators,
(and with control of the norms of the elements).
The number ten was improved to two by Marcoux~\cite{M06}.
Pearcy and Topping, in~\cite{PT69}, showed that in the type II$_1$ factors of Wright (which do not have separable predual),
every self--adjoint element of trace zero is a commutator.

In section~\ref{sec:normal}, we employ the construction of Pearcy and Topping for the Wright factors
and a result of Hadwin~\cite{Had98} to show firstly 
that all normal elements of trace zero in the Wright factors are commutators.
We then use this same construction to derive
that in any II$_1$--factor, every normal element with trace zero and purely atomic distribution is a single commutator. 
In section~\ref{sec:nilpotent}, we show that all nilpotent operators in II$_1$--factors are commutators.
Finally, in section~\ref{sec:ques}, we provide classes of examples of elements of
II$_1$--factors that are not normal and not nilpotent but are single commutators, and we ask some
specific questions suggested by our examples and results.

\smallskip
\noindent{\bf Acknowledgement.}
The authors thank Heydar Radjavi for stimulating discussions about commutators,
and Gabriel Tucci for help with his operators.

\section{Some normal operators}
\label{sec:normal}

The following lemma (but with a constant of $2$) was described
in Concluding Remark (1) of~\cite{PT69},
attributed to unpublished work of John Dyer.
That the desired ordering of eigenvalues can be made
with bounding constant~$4$ follows from work of Steinitz~\cite{St13}, the value $2$ follows from~\cite{GS80}
and
the better constant in the version below (which is not actually needed in our application of it)
is due to work of Banaszczyk~\cite{Ba87}, \cite{Ba90}.
\begin{lemma}\label{lem:Anormal}
Let $A\in M_n(\Cpx)$ be a normal element with $\tr_n(A)=0$.
Then there are $B,C\in M_n(\Cpx)$ with $A=[B,C]$ and $\|B\|\,\|C\|\le\frac{\sqrt5}2\|A\|$.
\end{lemma}
\begin{proof}
After conjugating with a unitary, we may without loss of generality assume $A=\diag(\lambda_1,\ldots,\lambda_n)$
and we may choose the diagonal elements to appear in any prescribed order.
We have $A=[B,C]$ where
\begin{equation}\label{eq:B}
B=\left(
\begin{matrix}
0&1&0&\cdots&0 \\
0&0&1  \\
\vdots & &\ddots&\ddots \\
0& & \cdots    &  0   &1 \\
0&  0    & \cdots &      &0
\end{matrix}\right)
\end{equation}
and $C=B^*D$, where
\begin{equation}\label{eq:D}
D=\diag(\lambda_1,\,\lambda_1+\lambda_2,\,\ldots,\lambda_1+\cdots+\lambda_{n-1},0).
\end{equation}
By work of Banaszczyk~\cite{Ba87}, \cite{Ba90}, any list $\lambda_1,\ldots,\lambda_n$
of complex numbers whose sum is zero
can be reordered so that for all $k\in\{1,\ldots,n-1\}$ we have
\begin{equation}\label{eq:lambdasum}
\left|\sum_{j=1}^k\lambda_j\right|\le\frac{\sqrt5}2\max_{1\le j\le n}|\lambda_j|.
\end{equation}
This ensures $\|B\|\le1$ and $\|C\|\le\frac{\sqrt5}2\|A\|$.
\end{proof}

The II$_1$--factors of Wright~\cite{W54}
are the quotients of the von Neumann algebra of all bounded sequences in
$\prod_{n=1}^\infty M_n(\Cpx)$ by the ideal $I_\omega$, consisting
of all sequences $(a_n)_{n=1}^\infty\in\prod_{n=1}^\infty M_n(\Cpx)$
such that $\lim_{n\to\omega}\tr_n(a_n^*a_n)=0$, where $\omega$ is a nontrivial ultrafilter on the natural numbers.
The trace of the element of $\Mcal$ associated to a bounded sequence $(b_n)_{n=1}^\infty\in\prod_{n=1}^\infty M_n(\Cpx)$
is $\lim_{n\to\omega}\tr_n(b_n)$.
(See~\cite{McD70} or~\cite{J72} for ultrapowers of finite von Neumann algebras.)
The following result in the case of self--adjoint operators is due to Pearcy and Topping~\cite{PT69}.

\begin{thm}\label{thm:PT}
If $\Mcal$ is a Wright factor and if $T\in\Mcal$ is normal with $\tau(T)=0$, then $T\in\Comm(\Mcal)$.
\end{thm}
\begin{proof}
Let $T\in\Mcal$ be normal and let $X$ and $Y$ be the real and imaginary parts of $T$, respecitvely.
Let $(S_n)_{n=1}^\infty\in\prod_{n=1}^\infty M_n(\Cpx)$ be a representative of $T$, with $\|S_n\|\le\|T\|$ for all $n$.
Let $X_n$ and $Y_n$ be the real and imaginary parts of $S_n$.
Then the mixed $*$--moments of the pair $(X_n,Y_n)$ converge as $n\to\omega$ to the mixed $*$--moments of $(X,Y)$.
By standard methods, we can construct some commuting, self--adjoint, traceless $n\times n$ matrices
$H_n$ and $K_n$ such that $H_n$ converges
in moments to $X$ and $K_n$ converges in moments to $Y$, as $n\to\infty$.
Now using a result of Hadwin (Theorem 2.1 of~\cite{Had98}), we find $n\times n$ unitaries $U_n$ such that 
\[
\lim_{n\to\omega}\|U_nX_nU_n^*-H_n\|_2=0
\qquad
\lim_{n\to\omega}\|U_nY_nU_n^*-K_n\|_2=0,
\]
where $\|Z\|_2=\tr_n(Z^*Z)^{1/2}$ is the Euclidean norm resulting from the normalized trace on $M_n(\Cpx)$.
This shows that $T$ has respresentative $(T_n)_{n=1}^\infty$,
where $T_n=U_n^*(H_n+iK_n)U_n$ is normal and, of course, traceless.

By Lemma~\ref{lem:Anormal}, for each $n$ there are $B_n,C_n\in M_n(\Cpx)$ with $\|B_n\|=1$ and
$\|C_n\|\le\frac{\sqrt5}2\|T\|$ such that
$T_n=[B_n,C_n]$.
Let $B,C\in\Mcal$ be the images (in the quotient
$\prod_{n=1}^\infty M_n(\Cpx)/I_\omega$) of $(B_n)_{n=1}^\infty$ and $(C_n)_{n=1}^\infty$, respectively.
Then $T=[B,C]$.
\end{proof}

The {\em distribution} of a normal element $T$ in a II$_1$--factor is the compactly supported Borel probability measure
on the complex plane obtained by composing the trace with the projection--valued spectral measure of $T$.

\begin{thm}\label{thm:normalhyp}
If $R$ is the hyperfinite II$_1$--factor and if $\mu$ is a compactly supported Borel probability measure on
the complex plane such that $\int z\,\mu(dz)=0$,
then there is a normal element $T\in\Comm(R)$ whose distribution is $\mu$.
\end{thm}
\begin{proof}
We will consider a particular instance of the construction from the proof of Theorem~\ref{thm:PT}.
Let $\Mcal$ be a factor of Wright, with tracial state $\tau$.
Let $L$ be the maximum modulus of elements of the support of $\mu$.
We may choose complex numbers $(\lambda^{(n)}_j)_{j=1}^n$ for $n\ge1$ such that the measures $\frac1n\sum_{j=1}^n\delta_{\lambda_j^{(n)}}$
converge in weak$^*$--topology
to $\mu$ and all have support contained inside the disk of radius $L$ centered at the origin
and such that $\sum_{j=1}^n\lambda^{(n)}_j=0$ for each $n$.
Let $T_n=\diag(\lambda^{(n)}_1,\ldots,\lambda^{(n)}_n)\in M_n(\Cpx)$
and let $T\in\Mcal$ be the element associated to the sequence $(T_n)_{n=1}^\infty$.
Then the distribution of $T$ is $\mu$.
By~\cite{Ba87}, \cite{Ba90}, we can order these $\lambda^{(n)}_1,\ldots,\lambda^{(n)}_n$ so that
$\big|\sum_{j=1}^k\lambda^{(n)}_j\big|\le\frac{\sqrt5}2\|T\|$ for all $1\le k\le n$.
Then, as in the proof of Lemma~\ref{lem:Anormal}, we have $T_n=[B_n,B_n^*D_n]$ where $B_n$ and $D_n$ are the $n\times n$
matrices $B$ and $D$ of~\eqref{eq:B} and~\eqref{eq:D}, respectively.
If $B,D\in\Mcal$ are the images in the quotient of the sequences $(B_n)_{n=1}^\infty$ and $(D_n)_{n=1}^\infty$, respectively,
then $T=[B,B^*D]$.
However, note that $B\in\Mcal$ is a unitary element such that $\tau(B^k)=0$ for all $k>0$.
Moreover, the set $\{B^kDB^{-k}\mid k\in\Ints\}$ generates a commutative von Neumann subalgebra $\Ac$ of $\Mcal$
and every element of $\Ac$ is the image (under the quotient mapping) of a sequence $(A_n)_{n=1}^\infty$ where each $A_n\in M_n(\Cpx)$
is a diagonal matrix.
Thus, the unitary $B$ acts by conjugation on $\Ac$, and, moreover, we have $\tau(AB^k)=0$ for all $A\in\Ac$ and all $k>0$.
Therefore the von Neumann subalgebra generated by $\Ac\cup\{B\}$ is a case of the group--measure-space construction,
$\Ac\rtimes\Ints$, and is a hyperfinite von Neumann algebra
by \cite{Co76} and can, thus, be embedded into the hyperfinite II$_1$--factor $R$.
\end{proof}

The above proof actually shows the following.
\begin{cor}
Given any compactly supported Borel probability measure $\mu$ on the complex plane
with $\int z\,\mu(dz)=0$, there is $f\in L^\infty([0,1])$ and
a probability-measure-preserving transformation $\alpha$ of $[0,1]$ such that the distribution of $f-\alpha(f)$ equals $\mu$
and the supremum norm of $f$ is no more than $\frac{\sqrt5}2$ times the maximum modulus of the support of $\mu$.
\end{cor}

\begin{thm}\label{thm:atomic}
If $\Mcal$ is any II$_1$--factor and $T\in\Mcal$ is a normal element whose distribution
is purely atomic
and with trace $\tau(T)=0$, then $T\in\Comm(\Mcal)$.
\end{thm}
\begin{proof}
$\Mcal$ contains a (unital) subfactor $R$ isomorphic to the hyperfinite II$_1$--factor.
By Theorem~\ref{thm:normalhyp}, there is an element $\Tt\in\Comm(R)$ whose distribution equals the distribution of $T$.
Since this distribution is purely atomic,
there is a unitary $U\in\Mcal$ such that $U\Tt U^*=T$.
Thus, $T\in\Comm(\Mcal)$.
\end{proof}

\section{Nilpotent operators}
\label{sec:nilpotent}

The von Neumann algebra $\Mcal$ is embedded in $B(\HEu)$ as a strong--operator--topology closed, self--adjoint subalgebra.
If $T\in\Mcal$, we denote the self--adjoint projection onto $\ker(T)$ by $\kerproj(T)$ and the self--adjoint projection onto
the closure of the range of $T$ by $\ranproj(T)$.
Both of these belong to $\Mcal$, and we have \[\tau(\kerproj(T))+\tau(\ranproj(T))=1\]

The following decomposition follows from the usual sort of analysis of subspaces that one does also
in the finite dimensional setting.
\begin{lemma}\label{lem:UT}
Let $\Mcal$ be a II$_1$--factor and let $T\in\Mcal$ be nilpotent, $T^n=0$.
Then there are integers
$n\ge k_1>k_2>\ldots>k_m\ge1$ and for each $j\in\{1,\ldots,m\}$ there are equivalent projections $f^{(j)}_1,\ldots,f^{(j)}_{k_j}$
in $\Mcal$ such that
\begin{enumerate}[(i)]
\item $f^{(j)}:=f^{(j)}_1+\cdots+f^{(j)}_{k_j}$ commutes with $T$,
\item $f^{(1)}+\cdots+f^{(m)}=1$,
\item the $k_j\times k_j$ matrix of $f^{(j)}T$ with respect to these projections $f^{(j)}_1,\ldots,f^{(j)}_{k_j}$ is strictly upper triangular.
\end{enumerate}
\end{lemma}
In other words, the lemma says that $T$ lies in a unital $*$--subalgebra
of $\Mcal$ that is isomorphic to $M_{k_1}(\Afr_1)\oplus\cdots\oplus M_{k_m}(\Afr_m)$
for certain compressions $\Afr_j$ of $\Mcal$ by projections, and the direct summand component of $T$ in each $M_{k_j}(\Afr_j)$ is
a strictly upper triangular matrix.
\begin{proof}
The proof is by induction on $n$.
The case $n=1$ is clear, because then $T=0$.
Assume $n\ge2$.
We consider the usual system $p_1,p_2,\ldots,p_n$ of pairwise orthogonal projections with respect to which $T$ is upper triangular:
\begin{align*}
p_1&=\kerproj(T), \\
p_j&=\kerproj(T^j)-\kerproj(T^{j-1}),\quad(2\le j\le n).
\end{align*}
Then we have
\begin{gather}
\tau(\ranproj(Tp_j))=\tau(p_j),\qquad(2\le j\le n), \label{eq:Tpj} \\
\ranproj(Tp_j)\le\kerproj(T^{j-1})=p_1+p_2+\cdots+p_{j-1},\qquad(2\le j\le n), \label{eq:Tpjle} \\
\ranproj(Tp_j)\wedge(p_1+ p_2+\cdots+p_{j-2})=0,\qquad(3\le j\le n). \label{eq:rpw}
\end{gather}
Indeed, for~\eqref{eq:Tpj}, it will suffice to show $\kerproj(Tp_j)=1-p_j$.
For this, note that if $p_j\xi=\xi$ and $T\xi=0$, then $\xi\in\ker T\subseteq\ker T^{j-1}$.
Since $p_j\perp\kerproj(T^{j-1})$, this gives $\xi=0$.
The relation~\eqref{eq:Tpjle} is clear.
For~\eqref{eq:rpw}, if $q:=\ranproj(Tp_j)\wedge\kerproj(T^{j-2})\ne0$, then
by standard techniques (see, e.g.,
Lemma~2.2.1 of~\cite{CD09}), we would have a nonzero projection
$r\le p_j$ such that $q=\ranproj(Tr)\le\kerproj(T^{j-2})$.
However, this would imply $r\le\kerproj(T^{j-1})$, which contradicts $p_j\perp\kerproj(T^{j-1})$.

Let
\begin{align*}
q_n&=p_n\,, \\
q_{n-j}&=\ranproj(T^jq_n),\qquad(1\le j\le n-1).
\end{align*}
Then we have
\begin{gather}
q_k=\ranproj(Tq_{k+1})\le p_1+\cdots+p_k,\qquad(1\le k\le n-1), \label{eq:Tqk} \\
q_k\wedge(p_1+\cdots+p_{k-1})=0,\qquad(2\le k\le n). \label{eq:qk}
\end{gather}
Now~\eqref{eq:Tpj} and~\eqref{eq:Tqk} together imply $\tau(q_k)=\tau(q_{k+1})$,
and from~\eqref{eq:qk} we have $\tau(q_1\vee\cdots\vee q_k)=k\tau(q_1)$.
Thus, we have pairwise equivalent and orthogonal projections $f_1,\ldots,f_n$ defined by
\begin{align*}
f_n&=q_n\,, \\
f_k&=(q_k\vee\cdots\vee q_n)-(q_{k+1}\vee\cdots\vee q_n),\qquad(1\le k\le n-1),
\end{align*}
$T$ commutes with $f:=f_1+\cdots+f_n$ and $Tf$ is strictly upper triangular when written as an $n\times n$ matrix with respect
to $f_1,\ldots,f_n$.
Moreover, we have $(T(1-f))^{n-1}=T^{n-1}(1-f)=0$
and the induction hypothesis applies to $T(1-f)$.
\end{proof}

\begin{prop}
Let $\Mcal$ be a II$_1$--factor.
Then $\Comm(\Mcal)$ contains all nilpotent elements of $\Mcal$.
\end{prop}
\begin{proof}
By Lemma~\ref{lem:UT}, we only need to observe that a strictly upper triangular matrix in $M_n(\Afr)$ is
a single commutator, for any algebra $\Afr$.
But this is easy:  if
\[
A=\left(
\begin{matrix}
0&a_{1,2}&a_{1,3}&\cdots&a_{1,n} \\
0&0      &a_{2,3}&\cdots&a_{2,n} \\
\vdots &       &\ddots&\ddots&\vdots \\
 &       &       &      &a_{n-1,n} \\
0 &       & \cdots &      &0
\end{matrix}\right),
\]
then $A=BC-CB$, where $B$ is the matrix in~\eqref{eq:B},
\begin{equation}%\label{eq:BC}
C=\left(
\begin{matrix}
0&0&\cdots&0 \\
0&c_{2,2}&\cdots&c_{2,n} \\
\vdots& &\ddots&\vdots \\
0&\cdots&0    &c_{n,n}
\end{matrix}\right),
\end{equation}
and where the $c_{i,j}$ are chosen so that
\begin{align*}
a_{1,j}&=c_{2,j}\,,\qquad (2\le j\le n), \\
a_{p,j}&=c_{p+1,j}-c_{p,j-1}\,,\qquad(2\le p<j\le n).
\end{align*}
\end{proof}

\section{Examples and questions}
\label{sec:ques}

\begin{example}
A particular case of Theorem~\ref{thm:atomic} is that if $p$ is a projection
(with irrational trace)
in any II$_1$--factor $\Mcal$, then $p-\tau(p)1\in\Comm(\Mcal)$.
We note that a projection with rational trace is contained in some unital matrix
subalgebra $M_n(\Cpx)\subseteq\Mcal$; therefore, the case of a projection with rational trace is an immediate application of Shoda's result.
\end{example}

\begin{ques}
In light of Theorem~\ref{thm:atomic}, it is natural to ask:  does $\Comm(\Mcal)$ contain all normal elements of $\Mcal$
whose trace is zero?
(Note that each such element is the limit in norm of a sequence of elements of the sort considered in Theorem~\ref{thm:atomic}.)
It is of particular interest to focus on normal elements that generate maximal self--adjoint
abelian subalgebras (masas) in $\Mcal$.
Does it make a difference whether the masa is singular or semi-regular?
(See~\cite{SS08}.)
\end{ques}

A particular case:
\begin{ques}
If $a$ and $b$ freely generate the group $\Fb_2$, let $\lambda_a$ and $\lambda_b$
be the corresponding unitaries generating the group von Neumann algebra $L(\Fb_2)$.
Do we have $\lambda_a\in\Comm(L(\Fb_2))$?
\end{ques}

Our next examples come from ergodic theory.
\begin{example}\label{ex:ergodic}
Let $\alpha$ be an ergodic, probability measure preserving transformation of a standard Borel
probability space $X$,  that is not weakly mixing.
Consider the hyperfinite II$_1$--factor $R$ realized as the crossed product $R=L^\infty(X)\rtimes_\alphat\Ints$
where $\alphat$ is the automorphism of $L^\infty(X)$ arising from
$\alpha$ by $\alphat(f)=f\circ\alpha$.
For $f\in L^\infty([0,1])$, we let $\pi(f)$ denote the corresponding element of $R$, and we write $U\in R$ for the
implementing unitary,
so that $U\pi(f)U^*=\pi(\alphat(f))$.
By a standard result in ergodic theory (see, for example, Theorem 2.6.1 of~\cite{P83}), there is an eigenfunction,
i.e.,
$h\in L^\infty(X)\backslash\{0\}$ so that $\alphat(h)=\zeta h$ for some $\zeta\ne1$;
moreover, all eigenfunctions $h$ of an ergodic transformation must have $|h|$ constant.
If $g\in L^\infty(X)$, then
\begin{align*}
[U\pi(g),\pi(h)]=U\pi\big(g\big(h-\alphat^{-1}(h)\big)\big).
\end{align*}
Since $h-\alphat^{-1}(h)$ is invertible, by making appropriate choices of $g$ we get
$U\pi(f)=[U\pi(g),\pi(h)]\in\Comm(R)$ for all $f\in L^\infty(X)$.
\end{example}

\begin{ques}
If $\alpha$ is a weakly mixing transformation of $X$ (for example, a Bernoulli shift),
then, with the notation of Example~\ref{ex:ergodic},
do we have $U\pi(f)\in\Comm(R)$ for all $f\in L^\infty(X)$?
\end{ques}

\begin{example}
Assume that $\alphat$ from Example \ref{ex:ergodic} has infinitely many distinct eigenvalues. This is the case for every compact ergodic action $\alpha$ (for example, an irrational rotation of the circle or the odometer action), but can also hold for a non-compact action (for example, a skew rotation of the torus). For every finite set $F\subset\Ints\setminus\{0\}$, there is an eigenvalue $\zeta$ such that $\zeta^k\neq 1$, for any $k\in F$. Let $h$ be an eigenfunction of $\alphat$ corresponding to this eigenvalue $\zeta$; clearly, $|h|$ is a constant.
Then, for $g_k\in L^\infty(X)$,
\[
\left[\sum_{k \in F}U^k\pi(g_k),\pi(h)\right]=\sum_{k\in F} \left[U^k\pi(g_k),\pi(h)\right]
=\sum_{k\in F}U^k\pi\big(g_k\big(h-\alphat^{-k}(h)\big)\big).
\]
Thus, for any $f_k\in L^\infty(X)$, 
by choosing $g_k=f_k \big(h-\alphat^{-k}(h)\big)^{-1}$, we obtain 
\[
\sum_{k\in F} U^k\pi(f_k)\in\Comm(R).
\]
\end{example}

\begin{ques}
It is natural to ask Question~\ref{qn:comm} in the particular case of quasinilpotent elements $T$ of $\Mcal$:  must they lie in $\Comm(\Mcal)$?
From Proposition~4 of~\cite{MW79}, it follows that every 
quasinilpotent operator $T$ in a II$_1$--factor has trace zero.
(Alternatively, use L.\ Brown's analogue~\cite{B86} of Lidskii's theorem
in II$_1$--factors and the fact that the Brown measure of $T$ must be concentrated at $0$).
\end{ques}

\begin{ques}
Consider the quasinilpotent DT--operator $T$ (see~\cite{DH04}),
which is a generator of the free group factor $L(\Fb_2)$.
Do we have $T\in\Comm(L(\Fb_2))$?
\end{ques}

\begin{example}\label{ex:Tucci}
Consider G.\ Tucci's quasinilpotent operator
\[
A=\sum_{n=1}^\infty a_n V_n\in R,
\]
from~\cite{T08},
where $a=(a_n)_{n=1}^\infty\in\ell^1_+$, the set of summable sequences of nonnegative numbers.
Here $R=\overline{\bigotimes_1^\infty M_2(\Cpx)}$ is the hyperfinite II$_1$--factor and
\begin{equation}\label{def:Vn}
V_n=I^{\otimes n-1}\otimes\left(\begin{smallmatrix}0&1\\0&0\end{smallmatrix}\right)\otimes I\otimes I\otimes\cdots.
\end{equation}
Tucci showed in Remark~3.7 (p.\ 2978) of~\cite{T08} that $A$ is a single commutator whenever $a=(b_nc_n)_{n=1}^\infty$ for
some $b=(b_n)_{n=1}^\infty\in\ell^1$ and $c=(c_n)_{n=1}^\infty\in\ell^1$, by writing $A=[B,C]$, where
\begin{align}
B&=\sum_{n=1}^\infty b_nV_nV_n^*, \label{eq:Bop} \\
C&=\sum_{n=1}^\infty c_nV_n\,. \label{eq:C}
\end{align}
Note that, for $a\in\ell^1_+$,
there exist $b$ and $c$ in $\ell^1$ such that $a=(b_nc_n)_{n=1}^\infty$ if and only if
$\sum_{n=1}^\infty a_n^{1/2}<\infty$, i.e., if and only if $a\in\ell^{1/2}_+$.
\end{example}

The rest of the paper is concerned with some further results and remarks about Tucci's operators.

We might try to extend the formula $A=[B,C]$ for $B$ and $C$ as in~\eqref{eq:Bop} and~\eqref{eq:C}, respectively,
to other sequences $a\in\ell^1_+$, i.e.\ for $b$ and $c$ not necessarily in $\ell^1$, and where
the convergence in~\eqref{eq:Bop} and~\eqref{eq:C}
might be in some weaker topology.

We first turn our attention to~\eqref{eq:C}.
Denoting the usual embedding $R\hookrightarrow L^2(R,\tau)$ by $X\mapsto\Xh$,
from \eqref{def:Vn}
we see that the vectors $\Vh_n$ are orthogonal and all have $L^2(R,\tau)$-norm equal to
$1/\sqrt2$; therefore, the series~\eqref{eq:C}
converges in $L^2(R,\tau)$
as soon as $c\in\ell^2$, and we have
\begin{equation}\label{eq:Ch}
\Ch=\sum_{n=1}^\infty c_n\Vh_n.
\end{equation}
We easily see (below) that only for $c\in\ell^1$ there is a bounded
operator $C\in R$ such that $\Ch$ is given by~\eqref{eq:Ch}.

\begin{prop}\label{prop:cl1}
Let $c\in\ell^2$.
Suppose there is a bounded
operator $C\in R$ such that $\Ch$ is given by~\eqref{eq:Ch}.
Then $c\in\ell^1$.
\end{prop}
\begin{proof}
For any sequence $(\zeta_n)_{n=1}^\infty$ of complex numbers of modulus $1$, there is an automorphism of $R$ sending
$V_n$ to $\zeta_nV_n$ for all $n$.
Thus, without loss of generality we may assume $c_n\ge0$ for all $n$.

Letting $E_n:R\to M_2(\Cpx)^{\otimes n}\otimes I\otimes I\otimes\cdots\cong M_{2^n}(\Cpx)$ be the conditional expectation onto
the tensor product of the first $n$ copies of the $2\times 2$ matrices (see Example~\ref{ex:Tucci}),
we must have $C_n:=E_n(C)=\sum_{k=1}^n c_kV_k\in M_{2^n}(\Cpx)$.
Let $x=2^{-n/2}(1,1,\ldots,1)^t$ be the normalization of the column vector of length $2^n$ with all entries equal to $1$.
Taking the usual inner product in $\Cpx^{2^n}$, we see $\langle V_kx,x\rangle=1/2$ for all $k\in\{1,\ldots,n\}$.
Thus,
\[
\frac12\sum_{k=1}^nc_k=\big|\,\langle C_n x,x\rangle\,\big|\le\|C_n\|\le\|C\|.
\]
This shows $c\in\ell^1$.
\end{proof}

Let us now investigate the series~\eqref{eq:Bop} for some sequence $b=(b_n)_{n=1}^\infty$
of complex numbers.
We claim that this series gives rise (in a weak sense explained below)
to a bounded operator if and only if $b\in\ell^1$.
Indeed,
for $K$ a finite subset of $\Nats$, we have
\[
\left\|\sum_{n\in K}b_nV_nV_n^*\right\|_{L^2(R,\tau)}^2
=\;\frac14\sum_{n\in K}|b_n|^2+\frac14\left|\sum_{n\in K}b_n\right|^2.
\]
Now suppose $K_1\subseteq K_2\subseteq\cdots$ are finite sets
whose union is all of $\Nats$.
Then $\sum_{n\in K_p}b_nV_nV_n^*$ converges in $L^2(R,\tau)$ as $p\to\infty$
if and only if $b\in\ell^2$ and $y:=\lim_{p\to\infty}\sum_{n\in K_p}b_n$ exists.
Then the limit in $L^2(R,\tau)$ is
\begin{equation}\label{eq:Bh}
\Bh=\sum_{n=1}^\infty b_n\left(V_nV_n^*-\frac12\right)^{\widehat{\;}}+\frac y2 \oneh.
\end{equation}
If there is a bounded operator $B$ such that $\Bh$ is given by~\eqref{eq:Bh},
then for every finite $F\subseteq\Nats$, the conditional expectation $E_F(B)$
of $B$ onto the (finite dimensional)
subalgebra of $R$ generated by $\{V_nV_n^*\mid n\in F\}$ will be $\sum_{n\in F}b_n(V_nV_n^*-\frac12)+\frac y2$.
Taking the projection $P=\prod_{n\in F}V_nV_n^*$, we have $E_F(B)P=\frac12(y+\sum_{n\in F}b_n)P$, so
\[
\left|\frac12\left(y+\sum_{n\in F}b_n\right)\right|\le\|E_F(B)\|\le\|B\|.
\]
As $F$ was arbitrary, this implies $b\in\ell^1$.

Suppose $b_nc_n=\frac1{n^r}$ and $b=(b_n)_1^\infty\in\ell^1$.
Letting $(b^*_n)_1^\infty$ denote the nonincreasing rearrangement of $(|b_n|)_1^\infty$, we have $b^*_n=o(\frac1n)$
and standard arguments show $c^*_n\ge\frac{K}{n^{r-1}}$ for some constant $K$.
Thus, by Proposition~\ref{prop:cl1}, Tucci's formula for writing $A=[B,C]$ does not work if $a_n=\frac1{n^r}$ for $1<r\le2$,
while of course for $r>2$ it works just fine.

\begin{ques}\label{qn:Tuccir}
Fix $1<r\leq 2$, and let
\[
A=\sum_{n=1}^\infty \frac1{n^r}V_n\in R
\]
be Tucci's quasinilpotent operator in the hyperfinite II$_1$--factor.
Do we have $A\in\Comm(R)$?
\end{ques}

\end{document}